\documentclass[11pt]{article}

\usepackage[margin=1in]{geometry}
\usepackage[T1]{fontenc}
\usepackage[british]{babel}
\usepackage{amsmath,amssymb,amsthm,mathtools}
\usepackage[shortlabels]{enumitem}
\usepackage{microtype}
\usepackage{mathpazo}
\usepackage{xcolor}
\usepackage[
  colorlinks=true,
  linkcolor=red!55!black,
  citecolor=blue!60!black,
  urlcolor=red!55!black
]{hyperref}

\usepackage{authblk}

\newtheorem{theorem}{Theorem}[section]
\newtheorem{proposition}[theorem]{Proposition}
\newtheorem{lemma}[theorem]{Lemma}
\newtheorem{corollary}[theorem]{Corollary}
\theoremstyle{definition}
\newtheorem{definition}[theorem]{Definition}

\newtheorem{remark}[theorem]{Remark}

\title{Nef Hattori classes and products of projective lines}
\author[1]{Daniel \'Alvarez}
\author[2]{Andrew Clarke}
\author[3]{Eduardo Gonz\'alez}
\affil[1]{Instituto de Matematica, Universidade Federal do Rio de Janeiro, Brazil}
\affil[2]{Instituto de Matematica, Universidade Federal do Rio de Janeiro, Brazil}
\affil[3]{Department of Mathematics, University of Massachusetts  Boston, USA}

\date{1 October 2026}

\begin{document}
\maketitle

\begin{abstract}
Let \(X\) be a smooth connected complex projective variety with an effective algebraic \(\mathbb C^{\times}\)-action, a nonempty finite fixed locus whose weights are \(\pm1\). We show that if its Hattori line bundles are nef, then \(X\) is algebraically isomorphic to \((\mathbb{P}^1)^{\dim X}\), and these bundles become the hyperplane bundles of the factors. The product statement follows from the nef-complexity classification of Enwright and collaborators, and we describe the isomorphism through global sections. To test the required nefness we specialise the orbit-graph criterion of Romano and Wi\'sniewski. The Hirzebruch surface \(\mathbb F_2\) is an example why the topological hypotheses alone do not imply either the product conclusion or the Fano property.
\end{abstract}

\medskip
\noindent\textit{2020 Mathematics Subject Classification.} Primary 14L30. Secondary 14J45, 53D20.\par
\smallskip
\noindent\textit{Key words and phrases.} Semifree action, nef line bundle, nef complexity, Hattori class, Hirzebruch surface, and polarisation.\par

\section{Introduction}
\label{sec:introduction}

Let \(M\) be a compact connected symplectic manifold with a semifree Hamiltonian circle action, that is, one which is free away from its fixed points, and suppose that the fixed points are isolated and that there is at least one. Hattori showed that the integral cohomology ring and the Chern classes of \(M\) then agree with those of a product of two-spheres \cite{Ha-Sy92} (see also \cite[Theorem 1.2]{ToWe-On00}), and Tolman and Weitsman proved the equivariant version \cite[Theorems 2 and 1.1]{ToWe-On00}. It is natural to ask whether \(M\) must itself be a product of spheres, possibly with different symplectic areas. In \cite{ToWe-On00} it is pointed out that \((\mathbb{P}^1 )^2\) is the only example known in (real) dimension 4. In \cite{Go-Cl11} a proof of the six (real) dimension was provided. 

This paper is about the algebraic version of that question. Suppose that a smooth projective variety \(X\) of dimension \(n\) carries a semi-free algebraic \(\mathbb C^{\times}\)-action with isolated fixed points. Must it be isomorphic to a product \((\mathbb P^1)^n\) of complex projective lines? Now the complex structure has to be preserved, and the answer is no in general. The Hirzebruch surface \(\mathbb F_2\) has the desired semifree action and the cohomology ring of a product of lines, but its negative section prevents an algebraic isomorphism with \(\mathbb P^1\times\mathbb P^1\) see Section~\ref{sec:htw-package}.

What fails on \(\mathbb F_2\) is positivity, which relations in cohomology cannot detect. On a product, the line bundles pulled back from \(\mathcal O_{\mathbb P^1}(1)\), are nef, whereas on \(\mathbb F_2\) one of the corresponding classes has degree \(-1\) on the negative section. Once nefness is assumed, the classification theorem of Enwright, Filipazzi, Gongyo, Moraga, Svaldi, Wang, and Watanabe \cite[Theorem 4.2]{EnFiGoMoSvWaWa-Ne25} recovers the product.\footnote{We thank Dan Pomerleano for pointing out this result.} Our product theorem is an application of their result. We reduce to it, identify the line bundles involved, and describe the resulting isomorphism by global sections. To test nefness in practice we specialise the orbit-graph method of Romano and Wi\'sniewski \cite[Theorem 2.8]{RoWi-Ad22} to the Hattori classes.

For this paper all varieties are over \(\mathbb C\), curves are integral, and \(H^*(X;\mathbb Z)\) is singular cohomology. For a line bundle \(L\) and a curve \(C\subset X\) with normalisation \(\nu:\widetilde C\to C\), we write \(L\cdot C=\deg(\nu^*L)\), and \(L\) is \emph{nef} if this degree is non-negative for every curve.

Our first result does not involve the group action, and it is where the classification theorem is used. We write \(\operatorname{pr}_i\) for the projection of \((\mathbb P^1)^n\) onto its \(i\)-th factor.

\begin{theorem}
\label{thm:main}
Let \(X\) be a smooth connected projective variety of dimension \(n\geq1\), and let \(L_1,\ldots,L_n\) be nef line bundles. Put \(x_i=c_1(L_i)\in H^2(X;\mathbb Z)\). Suppose that the cohomology ring, top intersection, and first Chern class satisfy
\[
\begin{gathered}
H^*(X;\mathbb Z)\cong\frac{\mathbb Z[x_1,\ldots,x_n]}{(x_1^2,\ldots,x_n^2)},\\
\int_Xx_1\cdots x_n=1,\qquad c_1(TX)=2\sum_{i=1}^n x_i.
\end{gathered}
\]
Then there is an algebraic isomorphism
\[
\Phi:X\xrightarrow{\ \cong\ }(\mathbb P^1)^n,
\qquad
L_i\cong\Phi^*\operatorname{pr}_i^*
\mathcal O_{\mathbb P^1}(1)\quad(1\leq i\leq n).
\]
Each \(L_i\) is generated by its global sections and \(\dim_{\mathbb C}H^0(X,L_i)=2\). For any choice of basis \(s_{i,0},s_{i,1}\in H^0(X,L_i)\), the map
\[
p\longmapsto
\bigl([s_{1,0}(p):s_{1,1}(p)],\ldots,
[s_{n,0}(p):s_{n,1}(p)]\bigr)
\]
is an isomorphism with the property.
\end{theorem}

The cohomological hypotheses of Theorem~\ref{thm:main} are exactly what the following condition on the action provides.

\begin{definition}
\label{def:htw}
An action of \(\mathbb C^{\times}\) on a smooth connected projective variety satisfies the \emph{HTW condition} if it is effective, its fixed locus is nonempty and finite, and every weight of its tangent representation at every fixed point is \(1\) or \(-1\). \end{definition}

By the results of Hattori and of Tolman and Weitsman, recalled in Section~\ref{sec:htw-package}, an HTW variety satisfies the cohomological hypotheses of Theorem~\ref{thm:main} for suitable generators \(x_i\), and each \(x_i\) is the first Chern class of a unique line bundle \(L_i\). Having fixed the HTW generators, we call these the \emph{Hattori line bundles}. For an HTW variety the only remaining hypothesis of the product theorem is therefore nefness. We should say that in the language of Romano and Wi\'sniewski an HTW action is nothing but an \emph{equalised} action \cite[Definition 1.6]{RoWi-Ad22} whose fixed points are isolated. We keep the name HTW as a reminder of where the cohomological conditions come from.

The fixed points of an HTW action can be labelled \(p_I\) by the subsets \(I\subseteq\{1,\ldots,n\}\). We draw a directed edge \(p\to q\) when some nonconstant orbit tends to \(p\) as \(t\to0\) and to \(q\) as \(t\to\infty\), and we call an edge \emph{minimal} if its endpoints are not joined by a directed path with two or more edges (see Section~\ref{sec:graph}). Our second result tests nefness on these minimal edges.

\begin{theorem}
\label{thm:graph}
Let \(X\) satisfy the HTW condition, with fixed points labelled as in Theorem~\ref{thm:htw}. Suppose that every minimal directed edge has the form
\[
p_I\longrightarrow p_{I\cup\{j\}},\qquad j\notin I.
\]
Then every Hattori line bundle is nef. The isomorphism in Theorem~\ref{thm:main} can be chosen to be equivariant for the action
\[
t\cdot([z_{1,0}:z_{1,1}],\ldots,[z_{n,0}:z_{n,1}])
=([tz_{1,0}:z_{1,1}],\ldots,[tz_{n,0}:z_{n,1}]).
\]
\end{theorem}

The orbit graph controls the degrees of line bundles on curves, and the classification theorem then turns the resulting nef decomposition of the anticanonical class into an isomorphism. The condition has to be imposed on minimal edges only, since already on \((\mathbb P^1)^n\) there are orbits along which several coordinates move at once.

Finally, the algebraic counterpart of giving the spheres different areas is a choice of polarisation, that is, of an ample line bundle. Writing \(\mathcal O(a_1,\ldots,a_n)=\bigotimes_i\operatorname{pr}_i^*\mathcal O_{\mathbb P^1}(a_i)\), Theorem~\ref{thm:main} gives an isomorphism of polarised varieties
\[
\left(X,\bigotimes_{i=1}^n L_i^{\otimes a_i}\right)
\cong
\left((\mathbb P^1)^n,\mathcal O(a_1,\ldots,a_n)\right)
\]
for all integers \(a_1,\ldots,a_n>0\). The integers \(a_i\) change the polarisation but not the underlying variety.

%%%%%%%%%%%%
%%%%%%%%%%%%

\section{Generalities on the HTW condition}
\label{sec:htw-package}

\subsection{Tangent weights}

We collect here the consequences of the HTW condition that we need. We in essence recast all known results into algebraic statements.

Let \(X\) be a smooth connected complex projective variety of dimension \(n\) with an algebraic action of \(\mathbb C^{\times}\), and let \(X^{\mathbb C^{\times}}\) be its fixed locus. At a point \(p\in X^{\mathbb{\mathbb{C}^\times}}\) the tangent space splits into one-dimensional representations,
\[
T_pX=V_{p,1}\oplus\cdots\oplus V_{p,n},
\qquad
t\cdot v=t^{w_{p,j}}v
\quad
(t\in\mathbb C^{\times},\ v\in V_{p,j}),
\]
and the integers \(w_{p,j}\) are the \emph{tangent weights} at \(p\). Under the HTW condition of Definition~\ref{def:htw} we write
\[
\lambda(p)=\#\{j:w_{p,j}=-1\}
\]
for the number of negative weights at a fixed point \(p\). This integer will index the fixed points below.

The restriction on the weights forces the action to be free away from the fixed points, and this is what lets us parametrise orbit closures by projective lines in Section~\ref{sec:graph}.

\begin{lemma}
\label{lem:semifree}
If the action satisfies the HTW condition, then it is free outside its fixed locus.
\end{lemma}

\begin{proof}
Let \(x\in X\) be a point which is not fixed. Its stabiliser is a finite cyclic group \(\boldsymbol\mu_d=\{t\in\mathbb C^{\times}:t^d=1\}\), and we must show that \(d=1\). Suppose that \(d>1\) and choose \(\zeta\in\boldsymbol\mu_d\) with \(\zeta\ne1\). As \(\mathbb C^{\times}\) is abelian, \(\zeta\) fixes the orbit of \(x\) pointwise, and hence its closure, which contains the fixed point
\[
p=\lim_{t\to0}t\cdot x.
\]
All tangent weights at \(p\) are \(\pm1\), so \(\zeta\) acts on \(T_pX\) with eigenvalues \(\zeta^{\pm1}\), neither of which is one. Since a finite-order automorphism can be linearised near a fixed point, \(p\) is an isolated point of \(X^{\zeta}\). This contradicts the fact that \(X^{\zeta}\) contains the orbit closure through \(p\).
\end{proof}

\subsection{Attracting cells}

For a fixed point \(p\) let
\[
X^+(p)=
\left\{x\in X:\lim_{t\to0}t\cdot x=p\right\}.
\]
By the Bia\l ynicki-Birula decomposition \cite[Theorems 4.3 and 4.4]{Bi-Some73}, \(X\) is the disjoint union of these locally closed sets, and in our situation \(X^+(p)\cong\mathbb A^{n-\lambda(p)}\). Since \(X\) is irreducible, exactly one cell is open and dense. We denote its fixed point by \(p_{\varnothing}\), so that \(\lambda(p_{\varnothing})=0\) and
\[
X^+(p_{\varnothing})\cong\mathbb A^n.
\]
In particular \(X\) is rational. The Hodge numbers \(h^{j,0}\) are birational invariants, so \(H^0(X,\Omega_X^j)=0\) and hence \(H^j(X,\mathcal O_X)=0\) for \(j>0\). The exponential sequence then gives an isomorphism
\[
c_1:\operatorname{Pic}(X)
\xrightarrow{\ \cong\ }H^2(X;\mathbb Z),
\]
and every integral degree-two class occurring below is the first Chern class of a unique line bundle. We could equally have quoted \cite[Proposition 1.9]{RoWi-Ad22} here, which gives rationality and \(\operatorname{Pic}(X)\cong H^2(X;\mathbb Z)\) for any action whose source is a point, together with a basis of \(\operatorname{Pic}(X)\) made of closures of codimension-one cells.

\subsection{Cohomology and fixed-point labels}

Under the HTW hypothesis it follows that the action induced by the maximal compact subgroup \(S^1\subset \mathbb{C}^\times\) is Hamiltonian, this is already in \cite[Theorem 1]{ToWe-On00}, and for K\"ahler manifolds it goes back to Frankel \cite{Fr-Fi59}, who showed that a holomorphic circle action with a fixed point is Hamiltonian. Since we want to use their cohomological results we briefly recall the reason. Averaging a K\"ahler form over \(S^1\) gives an invariant K\"ahler form \(\omega\). If \(\xi\) generates the circle action, Cartan's formula shows that \(\iota_\xi\omega\) is closed, and it is exact because \(H^1(X;\mathbb R)=0\). Writing \(\iota_\xi\omega=-d\mu\), the function \(\mu\) is a moment map, so the circle action is Hamiltonian and the results of \cite{Ha-Sy92,ToWe-On00} apply.

We use the Borel construction \(H^*_{\mathbb C^{\times}}(X;\mathbb Z)=H^*(E\mathbb C^{\times}\times_{\mathbb C^{\times}}X;\mathbb Z)\), which is a module over \(H^*(B\mathbb C^{\times};\mathbb Z)=\mathbb Z[u]\). We fix the sign of \(u\) so that a one-dimensional representation of weight \(m\) has equivariant first Chern class \(mu\). An \emph{equivariant lift} of a class in \(H^*(X;\mathbb Z)\) is a preimage under the forgetful map to ordinary cohomology, and \(\alpha|_p\in\mathbb Z[u]\) denotes the restriction of an equivariant class \(\alpha\) to a fixed point \(p\). For the action \(t\cdot[z_0:z_1]=[tz_0:z_1]\) on \(\mathbb P^1\), the natural linearisation of \(\mathcal O(1)\) has fibre weight zero at \([0:1]\) and \(-1\) at \([1:0]\) (linearisations are recalled in Section~\ref{sec:graph}). This fixes the signs in the following statement.

The cohomology and Chern-class identities are Hattori's theorem \cite{Ha-Sy92}, in the form recalled in \cite[Theorem 1.2]{ToWe-On00}. The compatible fixed-point labels and equivariant lifts follow from \cite[Theorems 2, 1.1, and 1.2]{ToWe-On00} are as follows.

\begin{theorem}
\label{thm:htw}
Let \(X\) satisfy the HTW condition. The following statements hold.
\begin{enumerate}[(i),noitemsep]
\item There are exactly \(\binom nk\) fixed points \(p\) with \(\lambda(p)=k\).  In particular, \(X\) has \(2^n\) fixed points.
\item The fixed points can be labelled by subsets \(I\subset\{1,\ldots,n\}\), and \(\lambda(p_I)=|I|.\)
\item There are integral classes \(x_1,\ldots,x_n\in H^2(X;\mathbb Z)\) such that
\[
H^*(X;\mathbb Z)
\cong
\frac{\mathbb Z[x_1,\ldots,x_n]}
{(x_1^2,\ldots,x_n^2)},
\qquad
\int_Xx_1\cdots x_n=1.
\]
\item The total Chern class is \(c(TX)=\prod_{i=1}^n(1+2x_i)\). In particular, the first Chern class is
\[
c_1(TX)=2(x_1+\cdots+x_n).
\]
\item With the generator \(u\) fixed above, the classes \(x_i\) have equivariant lifts \(a_i\in H^2_{\mathbb C^{\times}}(X;\mathbb Z)\) satisfying
\[
a_i|_{p_I}=
\begin{cases}
-u,&i\in I,\\
0,&i\notin I,
\end{cases}
\]
with the convention that \(p_{\varnothing}\) is the point with all tangent weights positive.
\end{enumerate}
\end{theorem}

The normalisation \(\int_Xx_1\cdots x_n=1\) also follows from (iv), since \(c_n(TX)=2^nx_1\cdots x_n\) and the Euler characteristic of \(X\) is \(2^n\), the number of cells. Theorem~\ref{thm:htw} is purely topological, however, and says nothing about whether the line bundles with first Chern classes \(x_i\) are nef.

\subsection{The Hirzebruch \(\mathbb{F}_2\)}

We workout an example explicitly to show that the topological information is unfortunately not sufficient for strong positivity. Let
\[
\mathbb F_2
=\mathbb P_{\mathbb P^1}
\bigl(\mathcal O_{\mathbb P^1}
\oplus\mathcal O_{\mathbb P^1}(2)\bigr)
\]
be the bundle of lines in \(\mathcal O_{\mathbb P^1}\oplus\mathcal O_{\mathbb P^1}(2)\). The following proposition shows that it has all the topological properties listed in Theorem~\ref{thm:htw}, while one of its Hattori classes fails to be nef.

\begin{proposition}
\label{prop:f2}
The surface \(\mathbb F_2\) admits a \(\mathbb C^{\times}\)-action satisfying the HTW condition, but
\[
\mathbb F_2\not\cong\mathbb P^1\times\mathbb P^1.
\]
Moreover, its cohomology has generators \(x_1,x_2\) satisfying all the ordinary-cohomology and Chern-class relations in Theorem~\ref{thm:htw}, while \(x_2\) is not nef.
\end{proposition}

\begin{proof}
Let \(\mathbb C^{\times}\) act on the base by \(t\cdot[z_0:z_1]=[tz_0:z_1]\), with fixed points \(0=[0:1]\) and \(\infty=[1:0]\) of tangent weights \(1\) and \(-1\), respectively. We give the trivial summand \(\mathcal O\) the constant fibre weight \(-1\), and we linearise \(\mathcal O(2)\) with fibre weights \(0\) over \(0\) and \(-2\) over \(\infty\). The induced action on the projective bundle has four fixed points, namely the summand lines in the two fixed fibres. Writing \(E=\mathcal O\oplus\mathcal O(2)\), the vertical tangent space at a line \(\ell\subset E_b\) is \(\operatorname{Hom}(\ell,E_b/\ell)\), whose weight is the weight of \(E_b/\ell\) minus that of \(\ell\). This gives the following table.
\[
\begin{array}{c|c|c|c}
\text{fixed point}
&\text{base weight}
&\text{vertical weight}
&\text{tangent weights}\\ \hline
(0,[\mathcal O])&1&0-(-1)=1&(1,1)\\
(0,[\mathcal O(2)])&1&-1-0=-1&(1,-1)\\
(\infty,[\mathcal O])&-1&-2-(-1)=-1&(-1,-1)\\
(\infty,[\mathcal O(2)])&-1&-1-(-2)=1&(-1,1).
\end{array}
\]
Every tangent weight is therefore \(\pm1\), and the action is effective because it is effective on the base, so the HTW condition holds.

Let \(S\subset\mathbb F_2\) be the section given by the subbundle \(\mathcal O(2)\), and let \(F\) be a fibre. The normal bundle of \(S\) is \(\operatorname{Hom}(\mathcal O(2),\mathcal O)\cong\mathcal O_{\mathbb P^1}(-2)\), so
\[
S^2=-2,
\qquad S\cdot F=1,
\qquad F^2=0.
\]
On \(\mathbb P^1\times\mathbb P^1\) every effective curve has class \(aH_1+bH_2\) with \(a,b\geq0\), and hence self-intersection \(2ab\geq0\). The curve \(S\) therefore shows that \(\mathbb F_2\) is not isomorphic to the product.

Writing \(S\) and \(F\) also for their classes in \(H^2(\mathbb F_2;\mathbb Z)\), the projective bundle formula gives
\[
H^*(\mathbb F_2;\mathbb Z)
\cong
\frac{\mathbb Z[S,F]}{(F^2,S^2+2SF)},
\]
and the classes \(x_1=F\) and \(x_2=S+F\) satisfy
\[
x_1^2=x_2^2=0,
\qquad
\int_{\mathbb F_2}x_1x_2=1.
\]
To find the canonical class, write \(K_{\mathbb F_2}=aS+bF\). The adjunction formula \cite[Chapter V, Proposition 1.5]{Hr-Al77} on a fibre gives \(K_{\mathbb F_2}\cdot F=-2\), so \(a=-2\). On \(S\cong\mathbb P^1\) it gives \((K_{\mathbb F_2}+S)\cdot S=-2\), so \(K_{\mathbb F_2}\cdot S=0\) and \(b=-4\). Hence \(-K_{\mathbb F_2}=2S+4F\) and
\[
c_1(T\mathbb F_2)=2(x_1+x_2).
\]
Since \(\mathbb F_2\) is a \(\mathbb P^1\)-bundle over \(\mathbb P^1\), its Euler characteristic is four, and as \(x_1x_2\) has integral one we get
\[
c(T\mathbb F_2)=1+2(x_1+x_2)+4x_1x_2=(1+2x_1)(1+2x_2).
\]
These are the same identities as on the product. However, the second class has negative degree on the section,
\[
x_2\cdot S=(S+F)\cdot S=-1,
\]
so \(x_2\) is not nef. The anticanonical class also has degree zero on this curve,
\[
(-K_{\mathbb F_2})\cdot S=(2S+4F)\cdot S=0.
\]
An ample line bundle has strictly positive degree on every curve. \end{proof}

Thus \(\mathbb F_2\) is not \emph{Fano}, i.e. \(-K_{\mathbb{F}_2}\) is ample. In particular, the HTW condition alone does not make a variety Fano.

\begin{remark}
One might instead ask whether it is enough for some positive combination of the Hattori classes to be ample. The same example shows that it is not. If \(a,b\in\mathbb Z\) with \(a>0\) and \(b>2a\), then the line bundle \(\mathcal O_{\mathbb F_2}(aS+bF)\) has first Chern class \(aS+bF=(b-a)x_1+ax_2\), a positive combination of the two HTW generators, and it is ample. To see this, put \(A=aS+bF\). Then \(A^2=2a(b-a)>0\) and \(A\cdot S=b-2a>0\). Any other curve \(C\), of class \(cS+dF\), satisfies \(c=C\cdot F\geq0\) because the fibre class is nef, and \(d-2c=C\cdot S\geq0\) because distinct integral curves on a surface meet non-negatively. As \((c,d)\neq(0,0)\), we get \(A\cdot C=ad+(b-2a)c>0\), and the Nakai-Moishezon criterion \cite[Theorem 1.2.23]{La-Po04} shows that \(A\) is ample. So the total class can be ample even though \(x_2\) is not nef, and the hypothesis in Theorem~\ref{thm:main} has to concern the individual factor classes.
\end{remark}

Nor can the obstruction be removed by choosing different generators. If \(v=ax_1+bx_2\) is primitive and \(v^2=2ab=0\), then \(v\) is one of \(\pm x_1,\pm x_2\), and the only pair of such classes with sum \(c_1(T\mathbb F_2)/2=x_1+x_2\) is \(x_1,x_2\), in one order. The non-nef class is therefore forced on us by the first Chern class.

\section{The classification theorem}
\label{sec:numerical}

In this section we state the classification theorem of Enwright and collaborators in the form we use, and we check its generalised-pair hypothesis in the smooth case.

\subsection{Numerical divisor classes}

We write \(D\equiv D'\) when two Cartier divisors have the same degree on every curve, \([D]\) for the numerical class of \(D\), \(N^1(X)_{\mathbb R}\) for the real vector space of numerical classes, and \(\rho(X)=\dim_{\mathbb R}N^1(X)_{\mathbb R}\) for the Picard number. Additive formulas involving line bundles refer to the corresponding divisor classes, and line bundles with equal first Chern classes are numerically equivalent. A rational Cartier divisor \(D\) is nef if \(D\cdot C\geq0\) for every curve \(C\), and big if a positive integral multiple of it is a big Cartier divisor.

Within the nef cone, bigness is detected by a single intersection number, and we will use this twice. For a rational Cartier divisor \(D\) on a smooth projective variety of dimension \(n\),
\refstepcounter{equation}
\[
\tag{\theequation}
\label{eq:nef-big}
D\text{ is nef and big}\quad\Longleftrightarrow\quad
D\text{ is nef and }D^n>0,
\]
see \cite[Theorem 2.2.16]{La-Po04} and \cite[Remark 2.2(2)]{EnFiGoMoSvWaWa-Ne25}. For example, on a product of \(n\geq2\) projective lines each factor hyperplane bundle is nef but not big, while their sum is big.

\subsection{The equality case}

Following \cite[Definition 1.2]{EnFiGoMoSvWaWa-Ne25}, the \emph{total index} of a smooth projective variety \(X\) with nef and numerically nonzero anticanonical class is
\[
\tau_X=
\sup\left\{\sum_{j=1}^r b_j:
\begin{array}{l}
-K_X\equiv\sum_{j=1}^r b_jD_j,\quad b_j\in\mathbb Q_{>0},\\
D_j\text{ nef Cartier},\quad D_j\not\equiv0
\end{array}\right\},
\]
and its \emph{nef complexity} is \(c_X=\dim X+\rho(X)-\tau_X\). The divisors are required to be Cartier, since rescaling them by positive rationals would otherwise change the coefficient sum.

We will use \cite[Theorem 4.2]{EnFiGoMoSvWaWa-Ne25}, which we restate in our notation. A generalised klt log Calabi-Yau pair \((Y,B+M)\) consists of an effective real boundary divisor \(B\) and a nef part \(M\) such that all generalised log discrepancies are positive and \(K_Y+B+M\equiv0\). We recall these notions after the statement, in the only case we need.

\begin{theorem}
\label{thm:nef-complexity}
Let \((Y,B+M)\) be a generalised klt log Calabi-Yau pair, where \(Y\) is a normal projective variety of dimension \(n\geq1\). Suppose that \(D_1,\ldots,D_r\) are numerically nonzero nef Cartier divisors on \(Y\), and that \(b_1,\ldots,b_r\) are positive real numbers satisfying
\[
M=\sum_{j=1}^r b_jD_j.
\]
Assume that \(B+M\) is big. Let \(W\) be the real span of the numerical divisor classes, and write \(\sigma\) for its dimension,
\[
W=\operatorname{span}_{\mathbb R}\{[D_1],\ldots,[D_r]\}
\subseteq N^1(Y)_{\mathbb R},
\qquad
\sigma=\dim_{\mathbb R}W.
\]
Then the coefficient sum satisfies
\[
\sum_{j=1}^r b_j\leq\min\{n+\sigma,2n\}.
\]
If this bound is attained, all the following conclusions hold.
\begin{enumerate}[(i),noitemsep]
\item The numerical span satisfies \(n+\sigma\leq2n\).

\item There are positive integers \(d_1,\ldots,d_{\ell}\), with \(\sum_{k=1}^{\ell}d_k=n\), and an algebraic isomorphism
\[
\Psi:Y\xrightarrow{\ \cong\ }\prod_{k=1}^{\ell}\mathbb P^{d_k}.
\]
Write \(\operatorname{pr}_k\) for projection onto the \(k\)-th factor.

\item The boundary vanishes, and the decomposition is numerically anticanonical,
\[
B=0,
\qquad
K_Y+\sum_{j=1}^r b_jD_j\equiv0.
\]

\item Each divisor represents a factor hyperplane bundle. More precisely, for every \(j\in\{1,\ldots,r\}\), there is an index \(k(j)\in\{1,\ldots,\ell\}\) such that
\[
\mathcal O_Y(D_j)
\cong
\Psi^*\operatorname{pr}_{k(j)}^*
\mathcal O_{\mathbb P^{d_{k(j)}}}(1).
\]
\end{enumerate}
\end{theorem}

We apply this theorem with vanishing boundary and coefficient sum \(2n\). The following Theorem restricts the hypothesis to the smooth case at hand and the special equality in the sum above.

\begin{theorem}
\label{thm:classification}
Let \(Y\) be a smooth connected projective variety of dimension \(n\geq1\). Let \(D_1,\ldots,D_r\) be nef Cartier divisors which are not numerically trivial, and let \(b_1,\ldots,b_r\in\mathbb Q_{>0}\). Suppose that
\[
-K_Y\equiv M:=\sum_{j=1}^r b_jD_j,
\qquad M\text{ is big},
\qquad \sum_{j=1}^r b_j=2n.
\]
Then \(Y\cong(\mathbb P^1)^n\).
\end{theorem}

For us the nef part of a generalised pair is already a divisor on \(Y\), which keeps the birational formalism to a minimum. Let \(Y\) be smooth, let \(B\) be an effective rational divisor, and let \(M\) be a nef rational Cartier divisor on \(Y\). The nef part is the Cartier closure of \(M\), that is, \(M_Z=\pi^*M\) for every projective birational morphism \(\pi:Z\to Y\) from a smooth variety \(Z\). The divisor \(B_Z\) is then defined by
\refstepcounter{equation}
\[
\tag{\theequation}
\label{eq:generalised-boundary}
K_Z+B_Z+M_Z=\pi^*(K_Y+B+M),
\]
and the \emph{generalised log discrepancy} of a prime divisor \(E\subset Z\) is \(a_E(Y,B+M)=1-\operatorname{coeff}_E(B_Z)\). The pair is \emph{generalised klt} if all these numbers are positive, including for divisors on \(Y\), and it is a \emph{generalised klt log Calabi-Yau pair} if moreover \(K_Y+B+M\equiv0\). These are the definitions of \cite[Definitions 2.3, 2.4, and 2.5]{EnFiGoMoSvWaWa-Ne25} in our special case.

When \(B=0\), substituting \(M_Z=\pi^*M\) in \eqref{eq:generalised-boundary} gives
\[
B_Z=\pi^*K_Y-K_Z.
\]
The Jacobian determinant of \(\pi\) is a nonzero section of \(K_Z\otimes\pi^*K_Y^{-1}\), so \(K_Z-\pi^*K_Y\) is effective and every coefficient of \(B_Z\) is non-positive. Hence
\[
a_E(Y,0+M)\geq1>0,
\]
and \((Y,0+M)\) is a generalised klt log Calabi-Yau pair as soon as \(K_Y+M\equiv0\). The coefficients of \(M\) play no role in this, as \(M\) is not a boundary, so a coefficient equal to \(2\) is harmless.

\begin{proof}[Proof of Theorem \ref{thm:classification}]
Let
\[
W=\operatorname{span}_{\mathbb R}([D_1],\ldots,[D_r])
\subseteq N^1(Y)_{\mathbb R},
\qquad \sigma=\dim_{\mathbb R}W.
\]
By the discussion above, \((Y,0+M)\) is a generalised klt log Calabi-Yau pair, so Theorem~4.2 of \cite{EnFiGoMoSvWaWa-Ne25} gives
\[
2n=\sum_{j=1}^r b_j\leq\min(n+\sigma,2n),
\]
and therefore \(\sigma\geq n\). Theorem \ref{thm:nef-complexity} above of now gives \(\sigma\leq n\), an isomorphism with the  product \(Y\cong\prod_{k=1}^{\ell}\mathbb P^{d_k}\), and identify each of the \(D_j\) with a factor hyperplane class. Every factor must occur among these classes, for otherwise \(\sum_j b_jD_j\) would have degree zero on a line in the missing factor, where \(-K_Y\) has degree \(d_k+1>0\). Consequently \(\ell=\sigma=n\), and since \(\sum_kd_k=n\) with every \(d_k\geq1\), all the \(d_k\) are one.
\end{proof}

\begin{remark}
Under a Fano hypothesis, Theorem~1.6 of the same paper gives the coefficient bound and the product conclusion directly, and in our application \(X\) is shown to be Fano before the classification is used (Lemma~\ref{lem:fano}). We keep the formulation via Theorem~4.2 because it needs only a nef and big anticanonical decomposition.
\end{remark}

%%%%%%%%%%%%%%%
\section{Proof of the nef product theorem}
\label{sec:main-proof}

Throughout this section \(X\), \(L_1,\ldots,L_n\), and \(x_1,\ldots,x_n\) satisfy the hypotheses of Theorem~\ref{thm:main}. We choose divisors \(D_i\) with \(\mathcal O_X(D_i)\cong L_i\), which are unique up to linear equivalence, and put
\[
D=D_1+\cdots+D_n.
\]

\subsection{Intersections and positivity}

The relations \(x_i^2=0\) determine all the intersection numbers of the \(D_i\), and the next lemma states the ones we need.

\begin{lemma}
\label{lem:intersections}
The divisors \(D_i\) are nef and numerically nonzero, their numerical classes are independent, and
\[
D^n=n!,\qquad -K_X\equiv2D.
\]
In fact \(\rho(X)=n\).
\end{lemma}

\begin{proof}
The \(D_i\) are nef because the \(L_i\) are. For \(1\leq i,j\leq n\), the relations \(x_i^2=0\) and \(\int_Xx_1\cdots x_n=1\) give
\refstepcounter{equation}
\[
\tag{\theequation}
\label{eq:dual-intersections}
D_i\prod_{k\ne j}D_k
=\begin{cases}1,&i=j,\\0,&i\ne j,\end{cases}
\]
where for \(n=1\) the empty product is the fundamental cycle \([X]\). Pairing a relation \(\sum_ia_i[D_i]=0\) with the one-cycle \(\prod_{k\ne j}D_k\) gives \(a_j=0\), so the classes \([D_i]\) are independent and in particular nonzero. Since \(x_1,\ldots,x_n\) is a basis of \(H^2(X;\mathbb Z)\), every line bundle \(L\) has \(c_1(L)=\sum_ia_ix_i\) with \(a_i\in\mathbb Z\), and then \(L\equiv\sum_ia_iD_i\). Thus the \([D_i]\) span \(N^1(X)_{\mathbb R}\) and \(\rho(X)=n\). In the expansion of \(D^n\) every monomial with a repeated index vanishes, and each of the remaining \(n!\) monomials equals one, so \(D^n=n!\). Finally,
\[
c_1(-K_X)=c_1(TX)=2\sum_i x_i
=c_1\!\left(\bigotimes_iL_i^{\otimes2}\right),
\]
which gives \(-K_X\equiv2D\).
\end{proof}

In particular the anticanonical class satisfies the bigness hypothesis of Theorem~\ref{thm:classification}.

\begin{lemma}
\label{lem:big}
The anticanonical divisor is nef and big.
\end{lemma}

\begin{proof}
As a sum of nef divisors, \(D\) is nef. Hence \(-K_X\equiv2D\) is nef, with \((-K_X)^n=2^nn!>0\) by Lemma~\ref{lem:intersections}, and \eqref{eq:nef-big} shows that it is big.
\end{proof}

Since the \(D_i\) are nef and form a numerical basis, we can even prove ampleness before using the classification.

\begin{lemma}
\label{lem:fano}
For positive integers \(a_1,\ldots,a_n\), the divisor \(A=\sum_i a_iD_i\) is ample. In particular, \(D\) and \(-K_X\) are ample, and \(X\) is Fano.
\end{lemma}

\begin{proof}
Choose an ample Cartier divisor \(H\) on \(X\). By Lemma~\ref{lem:intersections} there are integers \(h_i\) with \(H\equiv\sum_ih_iD_i\), and we choose a rational number \(\varepsilon>0\) with \(a_i-\varepsilon h_i>0\) for every \(i\). Then
\[
A\equiv\varepsilon H+N,
\qquad
N=\sum_i(a_i-\varepsilon h_i)D_i,
\]
where \(N\) is nef. An ample class plus a nef class is ample \cite[Corollary 1.4.10]{La-Po04}, and ampleness only depends on the numerical class \cite[Corollary 1.2.24]{La-Po04}, so \(A\) is ample. Taking every \(a_i=1\) shows that \(D\) is ample, and hence so is \(-K_X\equiv2D\).
\end{proof}

We can now apply the classification theorem.

\begin{proposition}
\label{prop:abstract-product}
There exists an algebraic isomorphism \(\psi:X\to(\mathbb P^1)^n\).
\end{proposition}

\begin{proof}
We apply Theorem~\ref{thm:classification} with \(r=n\), the divisors \(D_i\), and all \(b_i=2\). The \(D_i\) are nef and numerically nonzero by Lemma~\ref{lem:intersections}, the sum \(\sum_i2D_i\equiv-K_X\) is big by Lemma~\ref{lem:big}, and \(\sum_ib_i=2n=2\dim X\).
\end{proof}

\subsection{The factor line bundles}

The isomorphism \(\psi\) need not respect the order of the \(L_i\), and we fix this next. Write \(P=(\mathbb P^1)^n\) and \(H_j=\operatorname{pr}_j^*\mathcal O_{\mathbb P^1}(1)\), with first Chern class \(h_j\). Every line bundle on \(P\) is isomorphic to a unique \(\mathcal O(a_1,\ldots,a_n)\), and \(a_j\) is its degree on a curve along which only the \(j\)-th coordinate varies.

\begin{lemma}
\label{lem:identify-bundles}
After a permutation of the factors of \(P\),
\[
L_i\cong\psi^*H_i\qquad(1\leq i\leq n).
\]
\end{lemma}

\begin{proof}
Write
\[
B_i=(\psi^{-1})^*L_i
\cong\mathcal O(a_{i1},\ldots,a_{in}),
\qquad a_{ij}\in\mathbb Z.
\]
Since \(B_i\) is nef, \(a_{ij}\geq0\). For \(n\geq2\) the relation \(x_i^2=0\) becomes
\[
0=c_1(B_i)^2
=2\sum_{j<k}a_{ij}a_{ik}h_jh_k,
\]
and as the classes \(h_jh_k\) with \(j<k\) form a basis of \(H^4(P;\mathbb Z)\), at most one coefficient in each row is nonzero. This holds trivially for \(n=1\). The class \(x_i\) is primitive, since it belongs to an integral basis, so the nonzero coefficient equals one. Writing \(\sigma(i)\) for its position, we get \(B_i\cong H_{\sigma(i)}\). The map \(\sigma\) is injective because the \(x_i\) are distinct, hence it is a permutation, and permuting the factors of \(P\) gives the claim.
\end{proof}

\subsection{The section morphism}

The identification of each \(L_i\) with a factor bundle also describes the isomorphism explicitly, through the two-dimensional spaces of sections.

\begin{proposition}
\label{prop:sections}
Every \(L_i\) is globally generated with \(\dim_{\mathbb C}H^0(X,L_i)=2\). Any bases \(s_{i,0},s_{i,1}\in H^0(X,L_i)\) define morphisms
\[
f_i:X\longrightarrow\mathbb P^1,
\qquad f_i(p)=[s_{i,0}(p):s_{i,1}(p)],
\]
whose product \(\Phi=(f_1,\ldots,f_n)\) is an isomorphism and satisfies \(\Phi^*H_i\cong L_i\).
\end{proposition}

\begin{proof}
Order the factors as in Lemma~\ref{lem:identify-bundles}. Since the other factors have only constant regular functions,
\[
H^0(P,H_i)=H^0(\mathbb P^1,\mathcal O(1))\cong\mathbb C^2,
\]
and \(H_i\) is globally generated. Pulling back by \(\psi\) gives the same properties for \(L_i\). Global generation means that \(s_{i,0}\) and \(s_{i,1}\) have no common zero, so \(f_i\) is a morphism. If the \(s_{i,a}\) are pulled back from the coordinate sections on \(P\), then \(f_i=\operatorname{pr}_i\circ\psi\) and \(\Phi=\psi\). Any other choice of bases differs by matrices \(A_i\in\operatorname{GL}_2(\mathbb C)\), and if \(a_i\) is the induced automorphism of \(\mathbb P^1\), the new map is
\[
(a_1\times\cdots\times a_n)\circ\psi.
\]
This is again an isomorphism, and since projective automorphisms preserve \(\mathcal O(1)\), it satisfies \(\Phi^*H_i\cong L_i\).
\end{proof}

\begin{proof}[Proof of Theorem \ref{thm:main}]
This is the combination of Propositions~\ref{prop:abstract-product} and \ref{prop:sections}.
\end{proof}

Taking positive tensor powers of the \(L_i\) gives the polarised version stated in the introduction.

\begin{corollary}
\label{cor:polarised}
For integers \(a_1,\ldots,a_n>0\), the line bundle \(A=\bigotimes_iL_i^{\otimes a_i}\) is ample, and
\[
(X,A)\cong((\mathbb P^1)^n,\mathcal O(a_1,\ldots,a_n)).
\]
Moreover \(X\) is Fano.
\end{corollary}

\begin{proof}
Composing the Veronese embeddings of the factors given by \(\mathcal O_{\mathbb P^1}(a_i)\) with the Segre embedding shows that \(\mathcal O(a_1,\ldots,a_n)\) is very ample, and pulling back by \(\Phi\) proves the first assertion. Since \(K_{\mathbb P^1}=\mathcal O(-2)\), we have \(-K_P=\mathcal O(2,\ldots,2)\), which is ample, so \(X\cong P\) is Fano.
\end{proof}

The Fano property was also obtained in Lemma~\ref{lem:fano}. Both arguments use nefness, and Proposition~\ref{prop:f2} shows that the conclusion fails without it.

%%%%%%%%%%%
%%%%%%%%%%%
\section{The orbit-graph criterion}
\label{sec:graph}

In this section \(X\) satisfies the HTW condition, with the generators, lifts, and labels of Theorem~\ref{thm:htw}, and \(L_i\) is the unique line bundle with \(c_1(L_i)=x_i\). We specialise the description of the nef cone in \cite[Theorem 2.8]{RoWi-Ad22} to these line bundles, and read off their degrees on all curves from the fixed-point weights.

\subsection{Normalised fibre weights}

We now spell all the notation needed for clarity of the result, so that the criterion is purely self-contained. Linearisations are what allow fixed-point weights to compute degrees, so we recall them first.

\begin{definition}
A \(\mathbb C^{\times}\)-\emph{linearisation} of a line bundle \(L\to X\) is an algebraic action on its total space which covers the given action on \(X\), is linear on each fibre, and it is compatible with the group action multiplication. For a fixed point \(p\), the induced representation on the one-dimensional fibre is
\[
t\cdot v=t^{\mu_L(p)}v,\qquad v\in L_p,
\]
for a unique integer \(\mu_L(p)\), called its \emph{fibre weight}. Twisting by the character \(t\mapsto t^k\) means replacing the action on the total space by \(t*v=t^k(t\cdot v)\). This adds \(k\) to every fibre weight.
\end{definition}

Every line bundle has a linearisation, and the only freedom is a simultaneous shift of the fibre weights.

\begin{lemma}
\label{lem:linearisation}
Every line bundle on a normal projective variety with a \(\mathbb C^{\times}\)-action admits a linearisation. If the variety is connected, any two linearisations of the same line bundle differ by a character.
\end{lemma}

\begin{proof}
By Brion's linearisation theorem \cite[Theorem 2.14]{Br-Li15}, on a normal variety acted on by a connected linear algebraic group, a fixed tensor power of every line bundle is linearisable, the exponent being the stable exponent of the Picard group of the group. For \(\mathbb C^{\times}\) this exponent is one, since \(K[t,t^{-1}]\) is a principal ideal domain and so \(\operatorname{Pic}(\mathbb C^{\times}_K)=0\) for every extension field \(K\) of \(\mathbb C\). Thus the line bundle itself is linearisable.

Two linearisations differ by an invertible regular function \(r(t,p)\) on \(\mathbb C^{\times}\times X\). It is constant in \(p\), because a connected projective variety has only constant regular functions, and the group law gives \(r(st)=r(s)r(t)\) and \(r(1)=1\). Hence \(r(t)=t^k\) for some \(k\in\mathbb Z\).
\end{proof}

With this freedom, the fibre weights can be matched with the equivariant lifts of Theorem~\ref{thm:htw}. Write \(\mathbf 1_{\{i\in I\}}\) for the indicator, equal to one if \(i\in I\) and zero otherwise.

\begin{lemma}
\label{lem:normalise-weights}
Each \(L_i\) has a linearisation whose fibre weights are
\refstepcounter{equation}
\[
\tag{\theequation}
\label{eq:normalised-weights}
\mu_{L_i}(p_I)=-\mathbf 1_{\{i\in I\}}.
\]

\end{lemma}

\begin{proof}
Choose a linearisation by Lemma~\ref{lem:linearisation}, and let \(\widetilde x_i\) be the equivariant first Chern class of the associated line bundle on the Borel construction, an equivariant lift of \(x_i\). Since \(H^1(X;\mathbb Z)=0\) and \(H^*(B\mathbb C^{\times};\mathbb Z)\) vanishes in odd degrees, the Leray-Serre spectral sequence of the Borel fibration has no nonzero differentials into or out of total degree two, and it gives an exact sequence
\[
0\longrightarrow\mathbb Z u
\longrightarrow H^2_{\mathbb C^{\times}}(X;\mathbb Z)
\longrightarrow H^2(X;\mathbb Z)\longrightarrow0.
\]
As \(a_i\) and \(\widetilde x_i\) both lift \(x_i\), we have \(\widetilde x_i-a_i=k_iu\) for an integer \(k_i\). Twisting the linearisation by \(t\mapsto t^{-k_i}\) turns its equivariant first Chern class into \(a_i\). At a fixed point \(p_I\) this class restricts to \(\mu_{L_i}(p_I)u\), so the restriction formula in Theorem~\ref{thm:htw} gives \eqref{eq:normalised-weights}.
\end{proof}

\subsection{Degrees from fibre weights}

Let \(x\in X\) be a point which is not fixed, and let \(C\) be the closure of its orbit, with its reduced structure. By Lemma~\ref{lem:semifree} the orbit map \(t\mapsto t\cdot x\) identifies \(\mathbb C^{\times}\) with a dense open subset of \(C\), so the normalisation of \(C\) is \(\mathbb P^1\) and the orbit map extends to it,
\[
\nu:\mathbb P^1\longrightarrow C\subset X,
\qquad \nu([t:1])=t\cdot x\quad(t\ne0).
\]
The map \(\nu\) is equivariant for the action \(s[t:1]=[st:1]\), as one sees on the dense orbit, and the points
\[
p=\nu([0:1]),\qquad q=\nu([1:0])
\]
are fixed. The following is the weight-one case of \cite[Lemma 2.2]{RoWi-Ad22}, and for equalised actions it is exactly \cite[Corollary 2.3(c)]{RoWi-Ad22}, where it is also shown that orbit closures are smooth. Since this is how we compute degrees throughout, we include the short proof, if only to fix the sign.

\begin{lemma}
\label{lem:degree-weight}
For a linearised line bundle \(L\) on \(X\),
\refstepcounter{equation}
\[
\tag{\theequation}
\label{eq:degree-weight}
L\cdot C=\mu_L(p)-\mu_L(q).
\]
In particular, \(p\ne q\).
\end{lemma}

\begin{proof}
Write \(\nu^*L\cong\mathcal O_{\mathbb P^1}(d)\), so that \(L\cdot C=d\). The natural linearisation of \(\mathcal O(1)\) on \(\mathbb P^1\) has weights \(0\) and \(-1\) at \([0:1]\) and \([1:0]\), so \(\mathcal O(d)\) has a linearisation with weights \(0\) and \(-d\). By Lemma~\ref{lem:linearisation}, any other linearisation, in particular the pullback of that of \(L\), has weights \(k\) and \(k-d\) for some \(k\in\mathbb Z\). This proves \eqref{eq:degree-weight}. For an ample line bundle \(A\) with a linearisation it gives \(\mu_A(p)-\mu_A(q)=A\cdot C>0\), so \(p\ne q\).
\end{proof}

For the Hattori bundles the formula becomes very simple. By \eqref{eq:normalised-weights}, an orbit closure directed from \(p_I\) to \(p_J\) satisfies
\refstepcounter{equation}
\[
\tag{\theequation}
\label{eq:hattori-degree}
L_i\cdot C
=\mathbf 1_{\{i\in J\}}-\mathbf 1_{\{i\in I\}}.
\]

\subsection{Degeneration to invariant curves}

To test nefness it is enough to know the degrees on orbit closures, because every curve degenerates to a sum of them. The following proposition produces an effective representative using the action, and again this seems to be completely standard and well known in the field, we include it for completeness. This is adapted from  \cite{RoWi-Ad22}. The third author learnt this idea of degeneration using the Hilbert scheme in a conversation with Daniel Halpern-Leistner. In fact this whole project was started because of this relation of positivity with degeneration.

\begin{proposition}
\label{prop:invariant-cycles}
Let \(X\) be a smooth projective variety with a \(\mathbb C^{\times}\)-action. Every curve \(C\subset X\) is numerically equivalent to an effective invariant one-cycle. If the fixed locus is finite, each component of this cycle is the closure of a one-dimensional orbit.
\end{proposition}

\begin{proof}
Choose a very ample line bundle \(A\) with a linearisation (Lemma~\ref{lem:linearisation}), and let \(P_C\) be the Hilbert polynomial of \(C\) with respect to \(A\). The linearisation identifies \(t^*A\) with \(A\), so every translate \(t\cdot C\) has Hilbert polynomial \(P_C\), and the translated family over \(\mathbb C^{\times}\), being isomorphic to \(\mathbb C^{\times}\times C\), is flat. This gives a morphism
\[
\beta:\mathbb C^{\times}\longrightarrow
\operatorname{Hilb}_{P_C}(X),
\qquad t\longmapsto[t\cdot C].
\]
Since the Hilbert scheme is proper, \(\beta\) extends to \(\overline\beta:\mathbb A^1\to\operatorname{Hilb}_{P_C}(X)\), and pulling back the universal family gives a flat projective family \(\pi:\mathcal Z\to\mathbb A^1\) whose fibre \(Z_t\) is \(t\cdot C\) for \(t\ne0\). For \(s\in\mathbb C^{\times}\) the maps \(t\mapsto s\cdot\overline\beta(t)\) and \(t\mapsto\overline\beta(st)\) agree on \(\mathbb C^{\times}\), hence everywhere because the Hilbert scheme is separated, and evaluating at zero shows that \(Z_0\) is invariant.

Let \(L\) be any line bundle on \(X\). Its pullback to \(\mathcal Z\) is flat over \(\mathbb A^1\), so both \(\chi(Z_t,L|_{Z_t})\) and \(\chi(Z_t,\mathcal O_{Z_t})\) are independent of \(t\) \cite[Chapter III, Theorem 9.9]{Hr-Al77}. Let \([Z_0]_1=\sum_jm_jC_j\) be the fundamental one-cycle of the special fibre, with \(m_j\) the length of its structure sheaf at the generic point of \(C_j\). Riemann-Roch on the normalisations of the \(C_j\), together with additivity of the Euler characteristic, gives
\refstepcounter{equation}
\[
\tag{\theequation}
\label{eq:euler-degree}
\chi(Z_0,L|_{Z_0})-\chi(Z_0,\mathcal O_{Z_0})
=\sum_jm_j(L\cdot C_j),
\]
because sheaves supported at finitely many points, such as embedded points and the corrections coming from normalisation, contribute equally to both terms. The same identity at \(t=1\) gives \(L\cdot C=\chi(C,L|_C)-\chi(C,\mathcal O_C)\), and by the constancy just proved
\[
L\cdot C=\sum_jm_j(L\cdot C_j).
\]
As \(L\) was arbitrary, \(C\) is numerically equivalent to \([Z_0]_1\).

Finally, the connected group \(\mathbb C^{\times}\) cannot permute the finitely many components \(C_j\) nontrivially, so each of them is invariant. A component on which the action is trivial lies in the fixed locus, which is finite. Every other component contains a one-dimensional orbit, which is dense, so it is an orbit closure.
\end{proof}

It follows that nefness can be tested on orbit closures alone.

\begin{corollary}
\label{cor:nef-orbits}
If the fixed locus is finite, a line bundle \(L\) on \(X\) is nef if and only if its degree is non-negative on every one-dimensional orbit closure.
\end{corollary}

\begin{proof}
One implication is clear. Conversely, given a curve \(C\), write \(C\equiv\sum_jm_jC_j\) as in Proposition~\ref{prop:invariant-cycles}. Then \(L\cdot C=\sum_jm_j(L\cdot C_j)\geq0\), since each \(C_j\) is an orbit closure.
\end{proof}

\subsection{Minimal edges}

Following \cite[Definition 2.6 and the ensuing definition of minimal edges]{RoWi-Ad22}, we record the endpoints of orbits in a finite directed graph. Only its minimal edges will matter.

\begin{definition}
The \emph{directed orbit graph} \(\Gamma(X)\) has vertex set \(X^{\mathbb C^{\times}}\). It has an edge \(p\to q\) when there is a non-fixed point \(x\in X\) with
\[
\lim_{t\to0}t\cdot x=p,
\qquad \lim_{t\to\infty}t\cdot x=q.
\]
A directed path of length \(r\) is a sequence of vertices \(p_0,\ldots,p_r\) with an edge \(p_{k-1}\to p_k\) for \(1\leq k\leq r\). An edge \(p\to q\) is \emph{minimal} if no path of length at least two has the same endpoints.
\end{definition}

If \(A\) is an ample line bundle with a linearisation, Lemma~\ref{lem:degree-weight} gives
\[
p\to q\quad\Longrightarrow\quad\mu_A(p)>\mu_A(q),
\]
so \(\Gamma(X)\) is a finite graph without directed cycles. Directed paths therefore define a partial order on the fixed points, and this is what lets us replace edges by chains of minimal ones.

\begin{lemma}
\label{lem:paths}
Every directed edge can be replaced by a path consisting of minimal edges. Therefore, if every minimal edge has the form \(p_I\to p_{I\cup\{j\}}\), every edge \(p_I\to p_J\) satisfies \(I\subsetneq J\).
\end{lemma}

\begin{proof}
Write \(p\leq q\) if \(p=q\) or there is a directed path from \(p\) to \(q\). Concatenation gives transitivity and the absence of cycles gives antisymmetry, so this is a partial order. If \(q\) covers \(p\), that is, \(p<q\) with no vertex strictly in between, then a path from \(p\) to \(q\) has no intermediate vertex, so it is a single edge, and that edge is minimal. Conversely, the endpoints of a minimal edge form a cover relation, since an intermediate vertex would give a longer path. Given \(p<q\), a chain between them of maximal length exists because the graph is finite, and its consecutive elements are cover relations, hence joined by minimal edges. Under the stated condition each of these edges adds one element to the label, so the initial label is strictly contained in the final one.
\end{proof}

In other words, the hypothesis of Lemma~\ref{lem:paths} says that the minimal orbit graph, with its HTW labels, is a subgraph of the Hasse diagram of the Boolean lattice of subsets of \(\{1,\ldots,n\}\).

With the normalised fibre weights, the degree formula turns nefness of the Hattori bundles into a containment condition on labels.

\begin{proposition}
\label{prop:exact-criterion}
All the Hattori line bundles are nef if and only if every edge \(p_I\to p_J\) of \(\Gamma(X)\) satisfies \(I\subsetneq J\).
\end{proposition}

\begin{proof}
Suppose first that the line bundles are nef, and let \(C\) be an orbit closure for an edge \(p_I\to p_J\). If some \(i\) belonged to \(I\setminus J\), formula \eqref{eq:hattori-degree} would give \(L_i\cdot C=-1\). Hence \(I\subseteq J\), and \(I\ne J\) because the endpoints are distinct. Conversely, if every edge has nested labels, then by \eqref{eq:hattori-degree} each \(L_i\) has degree zero or one on every orbit closure, and Corollary~\ref{cor:nef-orbits} shows that it is nef.
\end{proof}

\subsection{Equivariance}

\begin{proof}[Proof of Theorem \ref{thm:graph}]
By Lemma~\ref{lem:paths}, every edge \(p_I\to p_J\) has \(I\subsetneq J\), so the \(L_i\) are nef by Proposition~\ref{prop:exact-criterion}. Together with Theorem~\ref{thm:htw}, this gives all the hypotheses of Theorem~\ref{thm:main}, and hence the product isomorphism of Proposition~\ref{prop:sections}.

To make it equivariant, give each \(L_i\) the linearisation of Lemma~\ref{lem:normalise-weights} and let \(\mathbb C^{\times}\) act on \(V_i=H^0(X,L_i)\) by
\[
(t\cdot s)(p)=t\cdot\bigl(s(t^{-1}\cdot p)\bigr),
\qquad t\in\mathbb C^{\times},\ s\in V_i,\ p\in X.
\]
At a fixed point, evaluation \(V_i\to(L_i)_p\) is equivariant, and it is surjective because \(L_i\) is globally generated. The fibre weights are zero at \(p_{\varnothing}\) and \(-1\) at \(p_{\{1,\ldots,n\}}\), so both weights occur in \(V_i\), and as \(\dim_{\mathbb C}V_i=2\) each occurs exactly once. Let \(s_{i,0}\) and \(s_{i,1}\) span the weight \(-1\) and weight zero subspaces, respectively. A section \(s\) of weight \(m\) satisfies
\[
s(t\cdot p)=t^{-m}\bigl(t\cdot s(p)\bigr),
\]
and in a local frame the action of \(t\) on the fibres multiplies both sections by the same factor, which cancels in homogeneous coordinates. Therefore
\[
[s_{i,0}(t\cdot p):s_{i,1}(t\cdot p)]
=[t s_{i,0}(p):s_{i,1}(p)],
\]
which is the required equivariance of the map in Proposition~\ref{prop:sections}.
\end{proof}

The equivariance argument only uses nefness of the Hattori line bundles, however that is established. It also fixes the labels. At \(p_I\) the section of weight \(-1\) vanishes when \(i\notin I\) and the section of weight zero vanishes when \(i\in I\), so the isomorphism sends \(p_I\) to the point whose \(i\)-th coordinate is \([1:0]\) for \(i\in I\) and \([0:1]\) otherwise.

\subsection{The two model graphs}

On \((\mathbb P^1)^n\) with the diagonal weight-one action, every strict inclusion \(I\subsetneq J\) gives an edge from \(p_I\) to \(p_J\). Indeed, the orbit of the point with coordinate \([1:0]\) for \(i\in I\), \([1:1]\) for \(i\in J\setminus I\), and \([0:1]\) otherwise has exactly these limits. Conversely, along any orbit each coordinate either stays at a fixed point or moves from \([0:1]\) to \([1:0]\), so all edges come from strict inclusions, and the minimal ones add a single element. For example, the curve
\[
\mathbb P^1\longrightarrow\mathbb P^1\times\mathbb P^1,
\qquad [z_0:z_1]\longmapsto([z_0:z_1],[z_0:z_1])
\]
is an orbit closure giving an edge from \(p_{\varnothing}\) directly to \(p_{\{1,2\}}\). This edge is not minimal, since a path through either singleton vertex has the same endpoints, and it is the reason the condition in Theorem~\ref{thm:graph} is imposed on minimal edges only.

On \(\mathbb F_2\), the negative section of Proposition~\ref{prop:f2} is itself an orbit closure. Its endpoints have tangent weights \((1,-1)\) and \((-1,1)\), so they are labelled by two distinct singletons in \(\{1,2\}\), which are incomparable, and this edge violates the condition of Proposition~\ref{prop:exact-criterion}. By Lemma~\ref{lem:paths} the minimal-edge criterion fails as well. If the Hattori bundles are ordered so that \(c_1(L_2)=S+F\), the failure is precisely
\[
L_2\cdot S=(S+F)\cdot S=-1.
\]

\end{document}